\documentclass{amsart}
	\usepackage[utf8]{inputenc}    
\usepackage{amssymb}
\usepackage{latexsym,amsmath,amsfonts,amscd}
\usepackage{amsthm}
\usepackage{algorithm2e}
\usepackage{subfig}
\usepackage{array}
\usepackage{csvsimple}
\usepackage{siunitx,booktabs}
\usepackage{xcolor}
\usepackage{comment}
\usepackage{enumerate}
\usepackage{mathtools}

\usepackage{todonotes}
\DeclareMathOperator*{\argmin}{argmin}

\newcommand{\nrm}[1]{\left\| #1 \right\|}
\newcommand{\mean}[1]{{\langle #1 \rangle}}

\newcommand{\polN}{\mathbb N}
\newcommand{\polR}{\mathbb R}

\newcommand{\bfx}{\boldsymbol{x}}

\newcommand{\bigo}[1]{{\mathcal{O}\left(#1\right)}}
\newcommand{\Hm}{{\mathring H_{\mathrm{per}}^{-1}}}
\newcommand{\Homz}{{\mathring{H}_{\mathrm{per}}^{1}}}
\newcommand{\Ho}{{{H}_{\mathrm{per}}^{1}}}
\newcommand{\lapm}{{\Delta_M}}
\newcommand{\HMm}{{\mathring{H}_{\mathrm{per},M}^{-1}}}
\newcommand{\nlapmi}{{(-\Delta_M)^{-1}}}

\newcommand{\HMo}{{\mathring{H}^1_{\mathrm{per},M}}}
\newcommand{\half}{\frac{1}{2}}
\newcommand{\slot}{{\ \cdot \ }}

\newcommand{\calB}{{\mathcal{B}}}

\newcommand{\normL}[1]{{\| #1 \|_{\mathcal{L}}}}
\newcommand{\normLinv}[1]{{\| #1 \|_{\mathcal{L}^{-1}}}}

\newcommand{\dx}{{\operatorname{d}\!\bfx}}
\newcommand{\x}{\bfx}

\newcommand{\vkp}{{v_{k+1}}}
\newcommand{\vk}{{v_{k}}}

\newcommand{\etak}{{\eta_{k}}}

\newcommand{\dG}{{\delta G}}
\newcommand{\dGtil}{{\delta \tilde{G}}}
\newcommand{\LL}{\mathcal{L}}
\newcommand{\LLinv}{{\mathcal{L}^{-1}}}

\newcommand{\norm}[2]{\left\| #1 \right\| _{#2}}

\newcommand{\pairing}[2]{\bigl\langle #1 , #2 \bigr\rangle}
\newcommand{\iprd}[2]{\left( #1 , #2 \right)}
\newcommand{\iprdL}[2]{\left( #1 , #2 \right)_\LL}

\newcommand{\HH}{\mathbb{H}}
\newcommand{\RR}{\mathbb{R}}
\newcommand{\NN}{\mathbb{N}}

\newcommand{\s}{{\beta}}	
\newcommand{\ins}{{\alpha}} 
\newcommand{\CdvB}{{C_{B1}}} 
\newcommand{\CddB}{{C_{B2}}} 

\theoremstyle{definition}
\newtheorem{thm}{Theorem}[section]
\newtheorem{prop}[thm]{Proposition}
\newtheorem{cor}[thm]{Corollary}
\newtheorem{lem}[thm]{Lemma}
\newtheorem{defn}[thm]{Definition}
\newtheorem{rmk}[thm]{Remark}

	\graphicspath{{./figures/}}

	\title[PPGD]{A perturbed preconditioned gradient descent method for the unconstrained minimization of composite objectives}
	
	\author{Jea-Hyun Park}
  \address[J.-H. Park]{Department of Mathematics, Milwaukee School of Engineering, Milwaukee, WI 53202, USA} 
  \email[J.-H. Park]{park@msoe.edu}

  \author{Abner J.~Salgado}
  \address[A.J.~Salgado]{Department of Mathematics, The University of Tennessee, Knoxville, TN 37996, USA} 
  \email[A.J.~Salgado]{asalgad1@utk.edu}

  \author{Steven M.~Wise}
  \address[S.M.~Wise]{Department of Mathematics, The University of Tennessee, Knoxville, TN 37996, USA} 
  \email[S.M.~Wise]{swise1@utk.edu}
	
	\date{\today}

\begin{document}
	\maketitle

	\begin{abstract}
We introduce a perturbed preconditioned gradient descent (PPGD) method for the unconstrained minimization of a strongly convex objective $G$ with a locally Lipschitz continuous gradient. We assume that $G(v)=E(v)+F(v)$ and that the gradient of $F$ is only known approximately. Our analysis is conducted in infinite dimensions with a preconditioner built into the framework. We prove a linear rate of convergence, up to an error term dependent on the gradient approximation. We apply the PPGD to the stationary Cahn-Hilliard equations with variable mobility under periodic boundary conditions. Numerical experiments are presented to validate the theoretical convergence rates and explore how the mobility affects the computation.
	\end{abstract}

	\section{Introduction}

	\subsection{Problem of interest}
	\label{sec:prob-of-interest}

Let $\HH$ be a Hilbert space and $E,F:\HH \to \RR$ be smooth. Define the composite objective
\begin{equation}
\label{eq:GisComposite}
  G(v) \coloneqq E(v) + F(v), \qquad \forall v \in \HH.
\end{equation}
We assume that $G$ is strongly convex with Lipschitz continuous derivative. We are interested in the following composite optimization problem: Find
\begin{equation}
	\label{eqn:minimization}
	\min_{v \in \HH} G(v).
\end{equation}
Our main source of difficulty and originality is that the derivative of $F$ is only known approximately. We cast our problem in the optimization language, but the main applications we have in mind are related to the numerical solution of nonlinear partial differential equations (PDEs) where, in many cases, the solution to a PDE minimizes an associated energy functional.

The structure of $G$ naturally suggests the use of gradient-based methods; see \cite{nesterov2018lectures, wise2020FASD, feng2017PSD, psw2021PAGD}. However, the restrictions on $F$ make this approach infeasible. In Section~\ref{sec:app-to-CHeqn} we consider a scenario where this framework naturally arises.

	\subsection{Literature}\label{sec:literature}

For basic results on smooth optimization we refer to \cite{nesterov2004intro}. For convex objectives with globally Lipschitz derivative, the \emph{gradient descent} method (\cite[p. 25 (1.2.9)]{nesterov2004intro}) converges to the minimizer with rate $\bigo{1/k}$; whereas the \emph{optimal} or \emph{accelerated gradient descent} method (\cite[p. 76 (2.2.6)]{nesterov2004intro}) converges with rate $\bigo{1/k^2}$, where $k$ is number of iterations. When the objective is, in addition, strongly convex, the convergence rates become linear: $\bigo{(1-\mu/L)^k}$ and $\bigo{(1-\sqrt{\mu/L})^k}$, respectively. Here $\mu$ and $L$ are the best strong convexity and Lipschitz constants, respectively, that can be found.

Composite optimization, i.e., minimizing an objective consisting of two terms with different structural features, is intimately related to inexact optimization, and explicit discussions date back at least to 1979 \cite{lions1979splitting}. The authors of~\cite{lions1979splitting},  inspired by \cite{peaceman1955numerical,douglas1956numerical}, studied certain \emph{splitting algorithms} for monotone operators.
In \cite{nesterov2007gradient} a composite objective $\phi=f+\Psi$ is assumed, where $f$ is smooth and convex and $\Psi$ is non-smooth, but \emph{simple}. This reference also addresses the strongly convex case and shows that the gradient method (with \emph{composite gradient mapping} \cite[p. 4 (2.6)]{nesterov2007gradient}), and its accelerated version are as efficient as the smooth objective case.

In a similar vein \cite{daspremont2008smooth} shows that the accelerated descent method of \cite{nesterov1983method} with \emph{noisy} gradient retains the same convergence rate up to the gradient error: $\bigo{1/k^2} + 3\delta$, where $k$ is the number of iterations and $\delta$ is the gradient error parameter. Reference \cite{schmidt2011convergence} studies convergence of inexact methods for composite, convex objectives with non-smooth terms. It shows that the \emph{proximal-gradient descent} method has convergence rate $\bigo{1/k}$ and the \emph{accelerated proximal-gradient} method $\bigo{1/k^2}$, respectively, up to an error term. Their analysis leads to a similar convergence rate without error terms if these satisfy suitable summability conditions. When the smooth part of the objective is strongly convex, the rates are as in the smooth case, provided one assumes suitable decay rates on the error.

References \cite{devolder2013exactness, devolder2013firstorder, devolder2013strongconv, devolder2013intermediate} discuss inexact optimization in a comprehensive way.  Their study includes convergence and error behaviors of the \emph{primal gradient descent} method (PGM) and \emph{fast gradient descent} method (FGM) with inexact \emph{oracle}. For both convex and strongly convex objectives, the rates for PGM are, up to oracle error, the same as those for PGM with exact oracle; see \cite[p. 4 Definition 1]{devolder2013firstorder} and \cite[p. 2 Definition 1]{devolder2013strongconv}, or \cite[Definition 4.1 and 5.1]{devolder2013exactness} for more details. These references also study how errors accumulate and are \emph{amplified} for FGM. The \emph{amplifying errors} may seem at odds with the result of \cite{daspremont2008smooth}, where the magnitude of errors caused by noisy gradients stays bounded by $3\delta$ as mentioned earlier.  For detailed discussions and comparisons, see \cite[pp. 133-135]{devolder2013exactness}.

In signal processing applications, the optimization of composite functions arise naturally as the objective often consists of a (smooth) fidelity term and a non-smooth one, usually related to deblurring or sparsity promotion; see, for instance, \cite{daubechies2003iterative}. This reference seems to be responsible for popularizing the so-called ISTA (\emph{iterative shrinkage-thresholding algorithm}), for sparse signal processing problems. In essence, this is a forward-backward splitting method. In \cite{beck2009FISTA}, the \emph{fast iterative shrinkage-thresholding algorithm} (FISTA), which is an accelerated version of ISTA, was developed.

Our work can be situated between many of the previous ones. It can be seen, for example, as the PGM with inexact oracle of \cite{devolder2013strongconv}, where the computation of the gradients introduces inexactness. It can also be seen from the perspective of \cite{schmidt2011convergence}. However, there, the inexactness aims to provide a general framework and is treated abstractly. On the other hand, we make the inexactness explicit and the analysis is done in settings that are suitable for the numerical solutions of PDEs, which make the results of the previous works not applicable, as they assume finite dimensions and a global Lipschitz condition.

We preemptively make some valid criticisms of our work. First, we introduce a \emph{uniform vanishing property} for perturbations to treat inexactness explicitly. This may seem too strong to be useful in applications. Nevertheless, this assumption fits the applications we have in mind, as we illustrate in Section~\ref{sec:app-to-CHeqn}. Second, we do not use an accelerated scheme. As existing results show, including our own (e.g., \cite{psw2021PAGD, psw2023benchmark}), it is safe to assume that, at least in the practical sense, \emph{once basic gradient methods are established, the accelerated version will certainly show improved convergence}, especially in the strong convex setting. Thus, we focus on perturbations and the gradient method. Accelerated methods involve several more complications and will be the target of future study.

	\subsection{Summarized contributions}
	\label{sec:contribution}

Our contributions are summarized as follows. 
	\begin{enumerate}[1.]
	\item
We propose what we call the \emph{perturbed preconditioned gradient descent} method (PPGD) for the composite optimization of objectives like those described in Section~\ref{sec:prob-of-interest} and give a convergence analysis of the method. Unlike most existing works, we work in infinite dimensions and only assume a \emph{local} Lipschitz condition on the gradient of the objective. In the numerical approximation of PDEs, global Lipschitz conditions are too restrictive to be useful, and this always requires a growing dimensionality of the problem. In addition, we include preconditioning explicitly into our analysis. In the infinite dimensional setting, this is crucial to establish norm equivalences, which is at the core of our convergence analysis. Finally, we discuss inexactness explicitly using the \emph{uniform vanishing property} for perturbations; see Section~\ref{sec:PPGD-method}. This clarifies the restrictions imposed on computations of the inexact gradients. Also, it is tailored to a common case: the approximate solution of subproblems.

	\item 
As a tool for convergence, we prove estimates \eqref{est:dream-prelim} and \eqref{est:dream}, which hold for strongly convex objectives with \emph{locally} Lipschitz gradients. These are a generalization of some of central estimates that hold for convex, globally Lipschitz smooth objectives. They play an essential role in the convergence proof of the PPGD method by allowing us to obtain an \emph{invariant set} for the iterates. For descent-type methods the extension from global to a local Lipschitz condition is not relevant. However, due to the inexactness of the gradient, PPGD is not a descent method. This is the significance of our extension.
		
	\item 
As an application, we provide a numerical study of the stationary Cahn-Hilliard (CH) equations with variable mobility subject to periodic boundary conditions. We investigate not only the convergence of our numerical solver, but we also discuss how our method allows us to use the \emph{fast Fourier transform} (FFT). Due to the variable mobility, this is not possible with a standard approach.
	\end{enumerate}

	\subsection{Organization}
This paper is organized as follows. In Section~\ref{sec:prelim}, we collect notation and preliminary notions that are the basis of our analysis and results. In Section~\ref{sec:PPGD-method}, we prove convergence of the PPGD method, where the uniform vanishing property for perturbations is introduced. In Section~\ref{sec:app-to-CHeqn}, the stationary CH equation is fit into the PPGD framework and we study convergence of the resulting method. Lastly, in Section~\ref{sec:numerics}, we present numerical experiments for the stationary CH equation and discuss their theoretical and numerical implications.

	\section{Preliminaries}
	\label{sec:prelim}

Throughout $\coloneqq$ denotes equality by definition. $(\HH, (\, \cdot\, ,\, \cdot\, )_\HH)$ is a real and separable Hilbert space, $\HH^*$ is its dual, and $\pairing{\phi}{v}$ is the duality pairing of $v\in\HH$ and $\phi\in\HH^*$.  All \emph{energies} or \emph{objectives} $\HH \to \RR$ are assumed continuously Fr\'echet differentiable. For normed spaces $X$ and $Y$, $\calB(X,Y)$ is the space of bounded linear maps $X \to Y$.


\subsection{Strongly convex, locally Lipschitz smooth functions}\label{sub:prelim-Lip-conv}


We begin by defining a preconditioner.

	\begin{defn}[preconditioner]
	\label{def:precon}
A \emph{preconditioner} is $\LL \in \calB(\HH,\HH^*)$ such that the bilinear form
	$
(u,v)_\LL \coloneqq \pairing{\LL u}{v} 
	$
	satisfies, for any $u,v\in\HH$,
	\begin{align*}
(u,v)_\LL =(v,u)_\LL ,
	\quad 
(u,v)_\LL \le C_2 \norm{u}{\HH} \norm{v}{\HH}, 
	\quad 
C_1 \norm{u}{\HH}^2 \le (u,u)_\LL	,
	\end{align*}	
for some $C_1, C_2>0$.	We call these properties \emph{symmetry}, \emph{continuity}, and \emph{coercivity} of the bilinear form, respectively.
	\end{defn}

 Throughout, we assume that we have a preconditioner $\LL : \HH \to \HH^*$ available. All metric notions will be understood up to norm equivalence.

	\begin{prop}[properties of $\LL$]
	\label{prop:LL-Hilbert}
Let $\LL \in \calB(\HH,\HH^*)$ be a preconditioner. Then, $( \, \cdot \, , \, \cdot \,)_\LL$ is an inner product on $\HH$ and its induced norm,
	$
	\norm{ \, \cdot \, }{\LL}
	$,
	is equivalent to $\norm{ \, \cdot \, }{\HH}$ 
	\[
\sqrt{C_1} \norm{v}{\HH}\le\norm{v}{\LL}\le\sqrt{C_2} \norm{v}{\HH}.
	\]
	Moreover, the bilinear form
	\[
	  (\phi,\psi)_{\LLinv} \coloneqq \pairing{\phi}{\LLinv \psi}, \qquad \forall \phi,\psi \in \HH^*,
	\]
	is an inner product on $\HH^*$, which makes it complete. The induced norm,
	$
	\norm{ \, \cdot \, }{\LLinv} 
	$,
	is equivalent to $\norm{ \, \cdot \, }{\HH^*}$: 
	\[
	\frac{1}{\sqrt{C_2}} \norm{\phi}{\HH^*} 
	\le 
	\norm{\phi}{\LLinv} 
	\le 
	\frac{1}{\sqrt{C_1}} \norm{\phi}{\HH^*},
	\]
and it coincides with the operator norm with respect to the $\LL$--norm, i.e.,  
	\[
\norm{\phi}{\LLinv} \coloneqq \sup_{0\ne v\in\HH} \frac{ \pairing{\phi}{v} }{ \|v\|_\LL }.
	\]
Lastly, $\norm{\LLinv \phi}{\LL}=\norm{\phi}{\LLinv}$, for all $\phi\in \HH^*$ and $
	\|\LL v\|_\LLinv=\|v\|_\LL$, for every $v\in \HH$.
	\end{prop}

We will call $( \, \cdot \, , \, \cdot \,)_\LL$ the \emph{$\LL$--inner product}. Notice that, if $\LL$ is a preconditioner, then its inverse is the Riesz map with respect to the $\LL$--inner product.

Next we define strong convexity, Lipschitz continuity, and Lipschitz smoothness.

\begin{defn}[strong convexity]
	Let $G:\HH\to\RR$ be Fr\'echet differentiable, and $\delta G$ be its Fr\'echet derivative. We say that $G$ is \emph{$\mu$--strongly convex} (with respect to the $\LL$--norm) if there exists a constant $\mu>0$ such that
	\begin{equation}
		\pairing{\dG(u)-\dG(v)}{u-v}\ge \mu	 \norm{u-v}{\LL}^2 \quad \forall \, u,v\in\HH ,
		\label{est:strconv-pair}
	\end{equation}
	or, equivalently,
	\begin{equation}
		G(v)\ge G(u) + \pairing{\dG(u)}{v-u} +\frac{\mu}{2} \norm{v-u}{\LL}^2 \quad \forall \, u,v\in\HH .	
		\label{est:low-trap}
	\end{equation}
	If one of these holds with $\mu=0$, then $G$ is convex.
\end{defn}

Condition \eqref{est:low-trap} is often called the \emph{lower quadratic trap}. 

	\begin{rmk}[existence and uniqueness]
If an objective $G:\HH\to\RR$ is $\mu$--strongly convex ($\mu>0$), then it has a unique minimizer. A necessary and sufficient condition for the minimum is that the gradient vanishes; see \cite{nesterov2004intro, nesterov2018lectures, ciarlet1989intro}.
	\end{rmk}	

\begin{defn}[Lipschitz smoothness]\label{def:norm-loc-lip}
    Let $G:\HH\to\RR$. If, for every $B\subset \HH$ that is bounded and convex, there is $L_B>0$ such that
	\begin{equation}
		\pairing{\dG(u)-\dG(v)}{u-v} \le L_B\norm{u-v}{\LL}^2 \quad \forall \, u,v\in B,
		\label{est:lip-pair}
	\end{equation}
	or, equivalently,
	\begin{equation}
		G(v)\le G(u) + \pairing{\dG(u)}{v-u} +\frac{L_B}{2} \norm{v-u}{\LL}^2 \quad \forall \, u,v\in B,
		\label{est:up-trap}
	\end{equation}
	then we say that $G$ is \emph{locally Lipschitz smooth} with respect to the $\LL$--norm. If the constant $L_B$ can be chosen independently of $B$, then $G$ is said to be \emph{globally Lipschitz smooth} with respect to $\LL$--norm.
	If $\dG$ satisfies 
	\begin{equation}
		\norm{\dG(u)-\dG(v)}{\LLinv} \le L_B\norm{u-v}{\LL} \quad \forall \, u,v\in B,
		\label{est:lip-conti}
	\end{equation}
	we say that \emph{$G$ has locally Lipschitz derivatives} or $\dG:\HH\to\HH^*$ is \emph{locally Lipschitz continuous} with respect to $\LL$--norm. If $\dG$ satisfies \eqref{est:lip-conti} and $L_B=L$ can be chosen independently of $B$, $\dG$ is \emph{globally Lipschitz continuous} with respect to the $\LL$--norm.
\end{defn}

Notice that \eqref{est:lip-conti} implies the local Lipschitz smoothness of $G$. Convex, \emph{globally} Lipschitz smooth functions satisfy \emph{dual versions} of \eqref{est:lip-pair} and \eqref{est:up-trap}; see \cite[p.~67, Thm.~2.1.5 (2.1.12)]{nesterov2018lectures} and \cite[p.~77, Thm.~2.1.12 (2.1.32)]{nesterov2018lectures}. However, to our knowledge, this is unknown for the case of \emph{locally} Lipschitz smooth functions. Therefore, we present \eqref{est:dual-lip} and \eqref{est:dream}, which precisely assert the claimed generalization. Later, \eqref{est:dream} will play an essential role in the convergence proof of the PPGD method. Furthermore, these estimates can be useful whenever the objective function has only \emph{locally} Lipschitz derivatives and the minimizing method is not necessarily of descent-type. For these results, we require \emph{strong} convexity of the objective.

	\begin{lem}[dual lower trap]
	\label{lem:dual-est-llip}
Let $G:\HH\to\RR$ be strongly convex and locally Lipschitz smooth with respect to the $\LL$--norm. Then, for any $B\subset \HH$ that is bounded and convex, there is $\hat L_B>0$, such that
	\begin{equation}\label{est:dream-prelim}
		G(u) +\pairing{\dG(u)}{v-u} +\frac{1}{2\hat L_B} \norm{\dG(v)-\dG(u)}{\LLinv}^2 \le G(v), \quad \forall u,v\in B.
	\end{equation}
	Consequently, we have
	\begin{equation}\label{est:dual-lip}
		\frac{1}{\hat L_B}\norm{\delta G(v)-\delta G(u)}{\LLinv}^2 \le \pairing{\delta G(v)-\delta G(u)}{v-u}, \quad \forall u,v\in B.
	\end{equation}	
\end{lem}
\begin{proof}
	We follow \cite[p. 67, Thm.~2.1.5 (2.1.12)]{nesterov2018lectures}.
	Fix $u\in B$ and define $\tilde{G}(v) \coloneqq G(v)-\pairing{\dG(u)}{v}$.
	Notice that $\tilde{G}$ and $G$ share common local Lipschitz and strong convexity constants. 
	Notice also that $\dG$ maps bounded sets to bounded sets. Indeed, fixing $\norm{\dG(u)}{\LLinv}$, we have, for any $v\in B$, 
	\begin{align*}
\norm{\dG(v)}{\LLinv} & \le \norm{\dG(v) - \dG(u)}{\LLinv} + \norm{\dG(u)}{\LLinv} 
	\\
& \le L_B \norm{v-u}{\LL} +\norm{\dG(u)}{\LLinv} 
	\\
&\le L_B  (\norm{v}{\LL}+\norm{u}{\LL}) +\norm{\dG(u)}{\LLinv} 
	\\
& = 2L_B  M_B + \norm{\dG(u)}{\LLinv},
	\end{align*}
where $L_B$ is the Lipschitz constant of $\dG$ and $M_B \coloneqq \sup_{w\in B}{\normL{w}}$.

Define now $M_B' \coloneqq \sup_{w\in B}{\normLinv{\dG(w)}}$. Then, we have
	\begin{equation}\label{misc07}
			\norm{\dGtil(v)}{\LLinv} =\norm{\dG(v)-\dG(u)}{\LLinv}
			\le \norm{\dG(v)}{\LLinv} + \norm{\dG(u)}{\LLinv}\le 2M_B'.
	\end{equation}
	
	Fix now $s_0>0$ and define
	\[ 
	  B' \coloneqq \left\{w\in\HH \, \middle| \, \norm{w}{\LL}\le M_B+2 s_0 M_B' \right\}.
	\]
  Notice that $B \subset B'$, and let $ L_B'$ be the Lipschitz continuity constant of $\dG$ associated with $B'$. Then, for $s \in (0, s_0)$,
	\[
		\norm{v-s\LLinv\dGtil(v)}{\LL}\le \norm{v}{\LL} +s\norm{\dGtil(v)}{\LLinv}  \le M_B + 2 s_0 M'_B,
	\]
	i.e., $v-s\LLinv\dGtil(v) \in B'$.
	
	Notice now that $\dGtil(u)=0$ and thus $\tilde{G}$ attains its minimum at $u$. Then, since $v,\ v-s\LLinv\dGtil(v)\in B'$, for $0<s\le \min\{s_0, 1/ L_B'\}$, estimate \eqref{est:up-trap} on $B'$ implies
	\begin{align*}
G(u) &-\pairing{\dG(u)}{u} 
	\\
&= \tilde{G}(u)
	\\
& \le \tilde{G}(v-s\LLinv\dGtil(v)) 
	\\
&= G(v-s\LLinv(\dG(v)-\dG(u))) - \pairing{\dG(u)}{v-s\LLinv(\dG(v)-\dG(u))}
	\\
&\le G(v)  -s\pairing{\dG(v)}{\LLinv(\dG(v)-\dG(u))} 
	\\
& \quad +\frac{s^2 L_B'}{2} \norm{\LLinv(\dG(v)-\dG(u))}{\LL}^2 -\pairing{\dG(u)}{v} 
	\\
& \quad +s\pairing{\dG(u)}{\LLinv(\dG(v)-\dG(u))}
		\\
&= G(v) -\pairing{\dG(u)}{v} -s \left( 1-  \frac{s L_B'}{2} \right) \norm{\dG(v)-\dG(u)}{\LLinv}^2.
	\end{align*}
	Rearranging, we arrive at
	\begin{equation*}
		G(u) +\pairing{\dG(u)}{v-u} +s \left( 1-  \frac{s L_B'}{2} \right) \norm{\dG(v)-\dG(u)}{\LLinv}^2 \le G(v).
	\end{equation*}
	By setting $\hat L_B$ so that $1/(2\hat L_B) = s(1-  \frac{s L_B'}{2})>0$, we have the desired result. 

	If $ s_0 \ge 1/L'_B$, we only need to require $s\le 1/L'_B$ in the above argument, and we have $\hat L_B =  L_B'$, by setting $s=1/ L_B'$. The condition $ s_0 \ge 1/L'_B$ can always be made possible since we are free to choose $s_0$, at the expense of having larger constants $L'_B=\hat L_B$; a bigger $s_0$ makes $B'$ also bigger. 
	
	Estimate \eqref{est:dual-lip} follows by adding \eqref{est:dream-prelim} to itself with $u$ and $v$ exchanged.
\end{proof}

\begin{lem}[dual trap]\label{lem:dream}
	Let $G$ be $\mu$--strongly convex ($\mu>0$) locally Lipschitz smooth with respect to the $\LL$--norm. Then, for any $B\subset\HH$ bounded and convex, there are $\CdvB,\CddB>0$ such that, for any $u,v\in B$, we have
	\begin{equation}\label{est:dream}
		\pairing{\delta G(v)-\delta G(u)}{v-u} \ge \CdvB\norm{v-u}{\LL}^2+\CddB\norm{\delta G(v)-\delta G(u)}{\LLinv}^2,
	\end{equation}
	where $\CdvB \coloneqq(\mu^2 + 2\hat L_B)/(4\hat L_B+4\mu)$, $\CddB \coloneqq 1/(\mu+\hat L_B)$, and $\hat L_B$ is the constant from \eqref{est:dual-lip} applied to $v \mapsto G(v)-\frac \mu 4 \norm{v}{\LL}^2$.
\end{lem}
\begin{proof}
  We adapt \cite[p.~77, Thm.~2.1.12 (2.1.32)]{nesterov2018lectures}. Let $\tilde{G}(v)=G(v)-\frac \mu 4 \norm{v}{\LL}^2$, which also satisfies the assumptions of Lemma~\ref{lem:dual-est-llip}. It is $(\frac \mu 2)$--strongly convex and $\delta\tilde{G}$ is locally Lipschitz smooth on $B$ with constant $L_B - \tfrac\mu2$. 
	Apply \eqref{est:dual-lip} to $\tilde{G}$ on $B$. Noting that $\delta \tilde{G}(v)=\dG(v)-\frac\mu 2\LL v$, we have
	\begin{align*}
& \hspace{-0.25in} \pairing{\delta G(v)-\delta G(u)}{v-u} -\frac \mu  2 \norm{v-u}{\LL}^2 
	\\
& =\pairing{\delta G(v)-\delta G(u) -\frac \mu 2\LL(v-u)}{v-u} 
		\\		
& \ge \frac{1}{\hat L_B}\norm{\delta G(v)-\delta G(u) -\frac \mu  2\LL(v-u)}{\LLinv}^2
	\\
& = \frac{1}{\hat L_B}\norm{\delta G(v)-\delta G(u)}{\LLinv}^2 + \frac{\mu^2}{4\hat L_B}\norm{v-u}{\LL}^2 -\frac{\mu}{\hat L_B}\pairing{\delta G(v)-\delta G(u)}{v-u}.
	\end{align*} 
	Rearranging, we obtain \eqref{est:dream}.
\end{proof}

\begin{rmk}[the size of constants]
	\label{rmk:CB1CB2}
Estimates of $\CdvB^{-1}$, $\CddB$, $\CdvB \CddB$ can be useful for,  e.g., step size restrictions. They can usually be obtained by elementary calculations. In particular, $\CdvB^{-1}\le 6$, $\CddB < 1/L_B$, $0 < \CdvB \CddB < 1$. Hence $0 < 1 - \CdvB \CddB < 1$, and $\CdvB \approx 1/2$ in practice.
\end{rmk}

For completeness, we recall that the preconditioned gradient descent (PGD) method with step size $\ins>0$. Given $u_0\in\HH$, compute, for $k\ge0$,
\begin{equation}\label{mthd:PGD-egy}
	u_{k+1} = u_k -\ins\LLinv \delta \tilde{F}(u_k).
\end{equation}

	\begin{thm}[geometric convergence of PGD]
	\label{thm:PGD-conv}
Let $\tilde{F}:\HH\to\RR$ be $\tilde{\mu}$--strongly convex and $\delta\tilde{F}$ be locally Lipschitz continuous with constant $\tilde{L}$. Assume that \eqref{mthd:PGD-egy} is used to minimize $\tilde{F}$. Then there is $\ins_0>0$ such that if $\ins \in (0, \ins_0]$
	\begin{equation*}
\norm{u_k-u}{\LL}\le \rho^k \norm{u_0 - u}{\LL},
	\end{equation*}
where $\rho <1$ and $u$ is the unique minimizer of $\tilde{F}$.
	\end{thm}
	\begin{proof}
See \cite{nesterov2018lectures, feng2017PSD, wise2020FASD}.
	\end{proof}

\section{Perturbed preconditioned gradient descent}\label{sec:PPGD-method}

Here we define our perturbed preconditioned gradient descent (PPGD) method and give a convergence result for it.

	\begin{defn}[PPGD]\label{def:PPGD-mthd}
Let $G:\HH\to\RR$ admit the decomposition \eqref{eq:GisComposite}. Given a preconditioner $\LL$, a set of solver parameters $\Theta$, and a family of gradient approximations $\left\{ A(\, \cdot \, ;\theta) \, \middle| \, \theta \in \Theta\right\} \subset \calB(\HH,\HH^*)$ we define the \emph{perturbed preconditioned gradient descent} (PPGD) method as the following iterative scheme. Given $v_0 \in \HH$ and $\s >0$, called the step size, compute
	\begin{equation}
	\label{mthd:PPGD}
\vkp = \vk - \s\LLinv(\dG(\vk) + \etak) \quad \, k\ge 0,
	\end{equation}
where, as in \eqref{def:perturbation} below,
	\begin{equation}\label{def:perturbation-PPGD}
\eta_k \coloneqq \eta_{v_k,\theta_k}\coloneqq   A(v_k;\theta_k) - \delta F(v_k),
	\end{equation}
is a perturbation of the residual at the $k$--th iteration.
	\end{defn}
	
Here $A$ is an approximation for $\delta F$. If $A = \delta F$, the perturbation $\eta_k$ above is zero, and the algorithm is just the preconditioned steepest descent method (PSD). To analyze the algorithm, we must quantify the size of perturbations and, for this, we introduce a uniform approximation property for $A$ with respect to $\delta F$. Later, in Section~\ref{sec:app-to-CHeqn}, we will show how this property is verified in applications.

	\begin{defn}[perturbations uniformly vanish]
	\label{def:unif-vanish}
Let $\Theta$ be a directed solver parameter set with ordering $\preceq$. Let 
	\[
\left\{ A(\, \cdot \, ; \theta) \, \middle| \, \theta \in \Theta \right\} \subset \calB(\HH,\HH^*)
	\]
be a family of approximation operators with respect to $\delta F$ and define the perturbation (error) with respect to this approximation as
	\begin{equation}
	\label{def:perturbation}
\eta_{v,\theta} \coloneqq  A(v;\theta) - \delta F(v) , \quad \forall \ v\in \HH.
	\end{equation}
We say that the \emph{perturbations uniformly vanish} iff, for any bounded $B\subset \HH$ and any $\epsilon > 0$, there is $\theta_0 = \theta_0(B) \in \Theta$ such that, for all $\theta_0 \preceq \theta$,
	\begin{equation*}
\sup_{v \in B} \norm{\eta_{v,\theta}}{\LLinv} \le \epsilon.
	\end{equation*}
	\end{defn}

In the previous definition, we think of $\Theta$ as the set of practically adjustable parameters in an approximation scheme. For instance, we may set $\Theta = \{(v_0, \ins, N)\} \subset \HH \times (0,\infty) \times \NN$, where $v_0$ is the initial guess, $\ins$ is the step size, and $N$ the number of iterations. The ordering may then be given by the Cartesian products of neighborhoods of the minimizer, zero, and infinity (in their respective spaces).

	\begin{thm}[invariant set]
	\label{thm:invariant-set}
Let $G$ admit the decomposition \eqref{eq:GisComposite} and $\{v_k\}_{k\ge1} \subset \HH$ be obtained via the PPGD method \eqref{mthd:PPGD} with preconditioner $\LL$, initial guess $v_0\in \HH$, and with $\left\{A(v;\theta) \, \middle| \, \theta\in \Theta \right\}$ such that perturbations, defined in \eqref{def:perturbation-PPGD}, uniformly vanish. Set $u \coloneqq \argmin_{v\in\HH}G(v)$. Let $\CdvB>0$ and $\CddB>0$ be the constants appearing in \eqref{est:dream}. Define
	\[
d_0 \coloneqq \norm{v_0 - u}{\LL}, \quad \s_0 \coloneqq  \min\{\CddB, \CdvB^{-1} \} , \quad \epsilon_0 \coloneqq \left(\frac{1}{\CdvB} +\frac{5\CddB}{4} \right)^{-1} d_0^2 \CdvB.
	\]
If the step size is chosen so that
	\begin{equation}
	\label{est:sz-cond-invariant-set}
\s \in (0, \s_0],
	\end{equation}
and the solver parameters $\{\theta_k\}_{k\ge0}$ are chosen such that the perturbations $\eta_{k}$ satisfy, as guaranteed by the assumed uniform vanishing property,
	\begin{equation}
	\label{est:pert-cond-invariant-set}
\norm{\eta_{k}}{\LLinv}^2 \le \epsilon_0.
	\end{equation}
Then, for all $k \in \polN$, the iterates of PPGD satisfy
	\begin{align}
	\label{obj:invariant-set}
\vk \in B \coloneqq \left\{ v\in\HH \, \middle| \, \norm{v-u}{\LL} \le d_0  \right\}.
	\end{align}
	\end{thm}
		
	\begin{proof}
  Set, for $k \in \polN_0$, $d_k \coloneqq \norm{v_k - u}{\LL}$. We will show inductively that $d_k \leq d_0$, with the base for induction being trivial. Fix $k \in \polN$ and suppose $\vk\in B$. The assumptions on $G$ allow us to invoke Lemma~\ref{lem:dream} over $B$ so that
	\begin{equation}\label{misc10}
		\pairing{\dG(\vk)}{\vk-u} \ge \CdvB d_{k}^2 + \CddB \norm{\dG(\vk)}{\LLinv}^2.
	\end{equation}
	In addition, we trivially have
	\begin{equation}\label{misc09}
		\pairing{\etak}{\vk -u}
		\le	\frac{1}{2\CdvB} \norm{\etak}{\LLinv}^2+  \frac{\CdvB}{2}d_{k}^2
	\end{equation}
	and
	\begin{equation}\label{misc08}
		\pairing{\dG(\vk)}{\LLinv\etak}
		\le 
		\norm{\dG(\vk)}{\LLinv}^2+ \frac{1}{4} \norm{\etak}{\LLinv}^2.
	\end{equation}
	
	Subtract $u$ from \eqref{mthd:PPGD} and take $\LL$--norms. The choice of $\s$, \eqref{misc10}, \eqref{misc09}, and \eqref{misc08} yield,
	\begin{equation}
	\label{misc13}
	  \begin{aligned}
      d_{k+1}^2 &=\norm{\vk-u -\s\LLinv \delta G(\vk)-\s\LLinv\etak}{\LL}^2
      \\
      &= d_{k}^2 + \s^2\norm{\delta G(\vk)}{\LLinv}^2 + \s^2\norm{\etak}{\LLinv}^2 -2\s\pairing{\dG(\vk)}{\vk-u}
	\\
& \quad - 2\s\pairing{\etak}{\vk-u} +\s^2\pairing{\dG(\vk)}{\LLinv\etak}
      \\
      &\le (1-2\s \CdvB) d_{k}^2 - \s(2\CddB -\s)\norm{\delta G(\vk)}{\LLinv}^2 + \s^2\norm{\etak}{\LLinv}^2
	\\
& \quad - 2\s\pairing{\etak}{\vk-u} +\s^2\pairing{\dG(\vk)}{\LLinv\etak}
      \\
      &\le \left(1-\s\CdvB \right)d_{k}^2 - 2\s\left(\CddB -\s\right)\norm{\delta G(\vk)}{\LLinv}^2
	\\
& \quad +\s\left(\frac{1}{\CdvB} +\frac 5 4 \s\right)\norm{\etak}{\LLinv}^2
      \\
      &\le \left(1-\s\CdvB \right)d_{k}^2  + \s \left(\frac{1}{\CdvB} +\frac{5\CddB}{4} \right)\norm{\etak}{\LLinv}^2,
	  \end{aligned}
	\end{equation}
	where $\s \le \CddB$ was used in the last step and $0<\s \le \CdvB^{-1}$ ensures $0 \le \left(1-\s\CdvB \right) <1$.

	Now, since the perturbations uniformly vanish, for the bounded set $B$, we can choose $\theta_k$ such that \eqref{est:pert-cond-invariant-set} holds.
	Thus, we continue estimate \eqref{misc13} by using \eqref{est:pert-cond-invariant-set} and $\s \le \CdvB^{-1}$
	\[
		d_{k+1}^2 
		\le \left(1 - \s\CdvB \right) d_k^2 
		+
		d_0^2 \s\CdvB
		\le \left(1-\s\CdvB \right)d_0^2 +d_0^2 \s\CdvB
		= d_0^2.
	\]
	Therefore, we have $d_{k+1} \in B$, which completes the proof.
\end{proof}

\begin{rmk}[step size]
	In view of Remark~\ref{rmk:CB1CB2}, the step size restriction \eqref{est:sz-cond-invariant-set} is comparable to the classic one: $\s \le 1/L_B$.
\end{rmk}

\begin{thm}[convergence of PPGD]\label{thm:PPGD-conv-abs}
  Let $G$ admit the decomposition \eqref{eq:GisComposite} and $\{v_k\}_{k\ge1} \subset \HH$ be obtained via the PPGD method \eqref{mthd:PPGD} with preconditioner $\LL$, initial guess $v_0\in \HH$, and with $\left\{A(v;\theta) \, \middle| \, \theta\in \Theta \right\}$ such that perturbations, defined in \eqref{def:perturbation-PPGD}, uniformly vanish.
	Set $u \coloneqq \argmin_{v\in\HH}G(v)$, $d_0=\norm{v_0 - u}{\LL}$, and $\rho=1-\mu\s \in (0,1)$.
  If the step size $\s$ satisfies \eqref{est:sz-cond-invariant-set}, then the best approximation, $z_k$, up to $k$--th iterates, i.e., $z_k=\argmin_{i=0}^{k} \normL{v_k-u}$, satisfies
	\begin{align}\label{est:PPGD-err-conv}
		\norm{z_k - u}{\LL}^2  \le \frac{d_{0}^2}{\mu\s} \rho^{k}  +  \epsilon d_0,
	\end{align}
	where the iterates $v_k$'s involve perturbations satisfying \eqref{est:pert-cond-invariant-set}. 
\end{thm}
\begin{proof}
	Note that, by Theorem~\ref{thm:invariant-set}, we have $\{v_k\}_{k=0}^\infty \subset B$, where $B$ was defined in \eqref{obj:invariant-set}. As a consequence the same Lipschitz constant, $L_B$, can be used for all iterates.
	
	The proof adapts the arguments of \cite{devolder2013exactness}. Set, for $k \in \polN_0$,
	\[
	  d_k=\norm{v_k - u}{\LL}, \qquad \Delta G_k\coloneqq G(v_k)-G(u)\ge 0, \qquad e_{k+1}=\pairing{\etak}{u - \vkp}.
	\]
	Observe that
	\begin{equation}\label{misc30}
		d_{k+1}^2 =  d_{k}^2 -\norm{\vkp - \vk}{\LL}^2 +2\iprd{\vkp -\vk}{\vkp - u}_\LL.
	\end{equation}
	Take the inner product of \eqref{mthd:PPGD} with $\vkp - u$ to obtain
	\begin{equation}\label{misc31}
		\iprdL{\vkp - \vk}{\vkp-u}=\s\pairing{\dG(v_k)}{u - \vkp}+\s\pairing{\etak}{u - \vkp}.
	\end{equation}
	We may then combine \eqref{est:up-trap}, \eqref{est:low-trap}, and $\dG(u) = 0$ to conclude
	\begin{equation}\label{misc35}
    \begin{aligned}
      \pairing{\dG(v_k)}{u - \vkp} &= \pairing{\dG(v_k)}{u - \vk} + \pairing{\dG(v_k)}{\vk - \vkp} 
      \\
      & \le G(u) - G(\vkp) - \frac{\mu}{2} d_{k}^2 + \frac{L_B}{2} \norm{\vkp - \vk}{\LL}^2.
		\end{aligned}
	\end{equation}

	Using \eqref{misc31} and \eqref{misc35} we may rewrite \eqref{misc30} as
	\begin{align*}
		d_{k+1}^2 &=  d_{k}^2 -\norm{\vkp - \vk}{\LL}^2 +2\s\pairing{\dG(v_k)}{u - \vkp}+2\s\pairing{\etak}{u - \vkp}
			\\
		&\le (1-\mu\s) d_{k}^2 +2\s\left(G(u) - G(\vkp)\right) -(1-\s L_B) \norm{\vkp - \vk}{\LL}^2 
		\\
			&+  2\s\pairing{\etak}{u - \vkp}.
	\end{align*}
	Note that from \eqref{est:sz-cond-invariant-set} and Remark~\ref{rmk:CB1CB2} we have $\s \le 1/L_B$. The previous inequality then implies the key estimate
	\begin{align}\label{misc37}
		d_{k+1}^2 \le \rho d_{k}^2 +  2\s e_{k+1} -2\s \Delta G_{k+1} ,
	\end{align}
	which, when iterated, implies
	\[ 
	  d_k^2 \leq \rho^{k} d_{0}^2 +  2\s \sum_{i=0}^{k-1} \rho^i e_{k-i}  - 2\s \sum_{i=0}^{k-1} \rho^i  \Delta G_{k-i}.
	\]
	Rearranging we obtain
	\begin{equation}\label{misc39}
		\sum_{i=0}^{k-1} \rho^i  \Delta G_{k-i} \le \sum_{i=0}^{k-1} \rho^i  \Delta G_{k-i} + \frac{d_{k}^2}{2\s} \le \rho^{k} \frac{d_{0}^2}{2\s} +  \sum_{i=0}^{k-1} \rho^i e_{k-i}.
	\end{equation}
	
	Define $z_k \subset \{v_i\}_{i=0}^k$ as
	\[ 
	  z_k \coloneqq \argmin_{i=1}^k \Delta G_{k-i}
	\]
  and use \eqref{est:low-trap} to estimate
  \[
    \frac\mu2 \norm{z_k - u}{\LL}^2 \le G(z_k)-G(u).
  \]
  This, together with the fact that, for any $v \in B$, $\langle \eta_i, u - v \rangle \leq \epsilon d_0$, yields
	\begin{equation*}
		\frac{\mu}{2}\norm{z_k - u}{\LL}^2\sum_{i=0}^{k-1} \rho^i   \le \sum_{i=0}^{k-1} \rho^i  \Delta G_{k-i} \le \rho^{k} \frac{d_{0}^2}{2\s} +  \epsilon d_0 \sum_{i=0}^{k-1} \rho^i ,
	\end{equation*}
	which readily implies \eqref{est:PPGD-err-conv}.
\end{proof}

\begin{rmk}[convergence]
  Regarding Theorem~\ref{thm:PPGD-conv-abs}:
	\begin{enumerate}[1.]
		\item The error estimate \eqref{est:PPGD-err-conv} is essentially the same as \cite[p.~10 Thm.~4]{devolder2013strongconv} or \cite[p.~156 Thm.~5.4]{devolder2013exactness}. The additional factor $d_0$ in the second term can be explained by the way we quantify inexactness.
 
	\item By modifying \eqref{est:pert-cond-invariant-set}, a linear convergence rate \emph{without error} may be obtained. For example, if 
	\begin{equation*}
		\norm{\eta_{k}}{\LLinv}^2
		\le \min\left\{
		\left(\frac{1}{\CdvB} +\frac{5\CddB}{4} \right)^{-1} d_0^2 \CdvB, \frac{\mu}{4d_0} d_k^2  \right\},
	\end{equation*}
	we obtain, instead of \eqref{misc37},
		\begin{equation*}
      d_{k+1}^2 \le \rho d_{k}^2 +  2\s e_{k+1} \le \left( 1-\frac{\mu\s}{2} \right)d_k^2.
	\end{equation*}
	This requires a linear decay of the perturbations, which aligns with \cite{schmidt2011convergence}. 
	\end{enumerate}
\end{rmk}

\section{Application: The stationary Cahn-Hilliard equations with variable mobility}\label{sec:app-to-CHeqn}

\subsection{Motivation}
The stationary CH equations with a variable mobility arise after time discretization of the CH equations along with so-called \emph{convexity splitting} \cite{eyre1998unconditionally}. 
These equations are an important model that is used in many physical and biological phenomena, including spinodal decomposition \cite{cahn1961spinodal}, diblock copolymer \cite{choksi2009phase}, two-phase fluid flows \cite{badalassi2003computation}, and tumor growth \cite{wise2008three}, to name a few.
Its numerical implementation, however, is challenging due to their high order and nonlinear structure. Such a challenge becomes greater when the \emph{mobility} is not constant. In particular, in the context of materials science applications, where one usually imposes periodic boundary conditions, a variable mobility prevents the use of FFT for fast inversion of the ensuing linear systems. Here we will show how the PPGD allows us to decompose the computational difficulties: the nonlinearity corresponds to $E$ in \eqref{eq:GisComposite} whereas $F$ encodes a linear, but challenging component stemming from the variable mobility.

\subsection{Notation and settings}
Let $d \in \{2,3\}$ and $\Omega = (0,\ell)^d$, where $\ell>0$. For conciseness we will write $L^2 \coloneqq L^2(\Omega)$ and
\[
  L^2_0 = \left\{ f \in L^2 \, \middle| \, \mean{f} = 0 \right\},
\]
where $\mean{f}=\frac{1}{L^d} \int_\Omega f(\x)\dx$.
The $L^2$-inner product and norm are $(\cdot,\cdot)$ and $\norm{ \, \cdot \, }{}$, respectively. $\Ho$ is the subspace of functions in $H^1(\Omega)$ that are periodic, and $\Homz = \Ho \cap L^2_0$.
Their inner-products are
\[
(u,v)_\Homz = (\nabla u, \nabla v) = \sum_{i=1}^d (\partial_i u , \partial_i v), \quad \text{and} \quad (u,v)_\Ho = (u,v) + (u, v)_\Homz,
\]
respectively. We recall the following Poincar\'e inequality: there is $C_P>0$ such that
\begin{equation}
	\label{eq:PoincareMeanZero}
	\| v \| \leq C_P \| v \|_{\Homz}, \qquad \forall v \in \Homz.
\end{equation}
Define
\begin{align}
	\Hm &\coloneqq \left\{\phi \in (\Ho)^* \, \middle| \, \langle \phi , 1 \rangle = 0  \right\},
	\quad
  \norm{\phi}{\Hm} \coloneqq  \sup_{0\ne  v\in\Homz} \frac{\langle \phi ,  v\rangle}{\norm{ v}{\Homz}} .
	\label{def:Hm}
\end{align}
Note that $(\Homz, \norm{\slot}{\Homz})^* \cong (\Hm, \norm{\slot}{\Hm})$.

The following result defines an inner product on $\Hm$. Its proof is standard.

\begin{prop}[Riesz isometry]\label{prop:Riesz}
  Given $\phi\in \Hm$, let $\mathsf{T}(\phi) \in \Homz$, solve
	\begin{equation}
		\iprd{\nabla \mathsf{T}(\phi)}{\nabla v} = \langle \phi,  v\rangle \qquad \forall  v\in \Homz.
		\label{eqn:poisson-weak}
	\end{equation}
	This defines $\mathsf{T} \in \calB(\Hm,\Homz)$. In addition, the bilinear form 
  \begin{equation}
    \left(\phi,\xi\right)_{\Hm} \coloneqq \iprd{\nabla \mathsf{T}(\phi)}{\nabla\mathsf{T}(\xi)} =\langle\phi,\mathsf{T}(\xi)\rangle  = \langle \xi, \mathsf{T}(\phi)\rangle.
    \label{obj:Hm-ip}
  \end{equation}
  defines an inner product on $\Hm$, and the induced norm coincides with $\norm{ \, \cdot \, }{\Hm}$.
\end{prop}

%

\begin{rmk}[distributional identification]
	\label{rmk:identification}
	Notice that $\Homz \hookrightarrow L^2_0 \hookrightarrow \Hm$ forms a Gelfand triple. This gives a meaning to linear combinations of elements in any of these three spaces.
\end{rmk}


\subsection{The stationary Cahn-Hilliard equations with variable mobility}

Let $M : \Omega \to \RR$ be such that there are $0 <M_1 \leq M_2$ for which
\begin{equation}\label{est:mob-bd}
	M_1 \le M(\bfx) \le M_2 , \quad \forall \, \bfx \in \Omega.
\end{equation}
Define
\[
  \Delta_M v \coloneqq  \nabla\cdot\left(M \nabla v \right), \quad \forall \, v\in
\Homz.    
\]
With this notation at hand, the problem we wish to solve is as follows: given $f\in L^2$ and $u_\star\in\Homz$, find  $u\in\Homz$, $w\in\Homz$, and $C\in\polR$ that satisfy
\begin{align}
	u - u_\star & = \Delta_M w,
	\label{eqn:CH1}
	\\
	w + C & = u^3 - \Delta u - f .
	\label{eqn:CH2}
\end{align}
These are the \emph{stationary nonlinear} CH equations with variable mobility; $M$ is the \emph{mobility}, $u$ is the \emph{density field} or \emph{phase variable}, and $w$ is the \emph{chemical potential}. The chemical potential can be eliminated from the problem. By writing
\[
w = -(-\Delta_M)^{-1}\left(u-u_\star\right),
\]
and substituting in \eqref{eqn:CH1} we obtain
\begin{equation}
	\mathcal{N}(u) \coloneqq  (-\Delta_M)^{-1}\left(u - u_\star\right)-C + u^3 - \Delta u  = f.
	\label{eqn:CH3}
\end{equation}
The constant $C$ appears because of the nonlinearity and can be eliminated in the weak formulation of the problem. 

The weak form of \eqref{eqn:CH3}, which reads,
\[
  \left((-\Delta_M)^{-1} (u- u_\star), \psi  \right) + (u^3,\psi) + \left(\nabla u,\nabla\psi  \right) = (f,\psi), \quad  \forall \, \psi\in \Homz
\]
is clearly the Euler-Lagrange equation for the strictly convex energy
\begin{equation}\label{obj:CH-egy}
	G(v) \coloneqq  \frac{1}{2}\nrm{v-u_\star}_{\HMm}^2 + \frac{1}{4}\nrm{v}_{L^4}^4 +\frac{1}{2} \nrm{ v }_{\Homz}^2 - (f,v),
	\qquad 
	\forall v \in \Homz,
\end{equation}
where
\begin{equation}
\label{obj:HMo-ip}
  \nrm{ \psi }_{\HMm} \coloneqq  \sqrt{\left((-\Delta_M)^{-1} \psi,\psi \right)_{L^2}}, \quad \forall \, \psi\in \mathring{L}^2.
\end{equation}

\subsection{PGD for a non-constant coefficient elliptic problem}
\label{sec:PGD-Pois}

Rewriting the stationary CH problem as \eqref{eqn:CH3} requires the inversion of a non-constant coefficient elliptic operator. Namely, given $M\in L^\infty$ that satisfies \eqref{est:mob-bd} and $\phi\in\Hm$, find $u_\phi \in\Homz$ such that
\begin{equation}
	(M\nabla u_\phi, \nabla v)_{} = \langle \phi, v\rangle,  \qquad \forall v\in\Homz.
	\label{eqn:lapm-weak}
\end{equation}

To approximate the solution to this problem, we employ the PGD method \eqref{mthd:PGD-egy} with the Laplacian as preconditioner. In this setting, given $u_0\in\Homz$, this method computes, for $n \ge 0$,
	\begin{align}
	r_n &= \phi + \Delta_M u_{n} ,
	\label{mthd:PGD-Pois-res}
	\\
	d_n &= (-\Delta)^{-1}r_n,
	\label{mthd:PGD-Pois-sd}
	\\
u_{n+1} & = u_{n} + \ins_n d_n .
	\label{mthd:PGD-Pois}
	\end{align}
where $\ins_n>0$ is the step size.

It is a trivial observation to note that
\begin{equation*}
  (u, v)_{\HMo}\coloneqq (M\nabla u, \nabla v)_{}
\end{equation*}
is an inner product on $\Homz$ and the induced norm satisfies, for all $u\in\Homz$,
\begin{equation}
  \sqrt{M_1} \|u\|_{\Homz} \le \|u\|_{\HMo} \le \sqrt{M_2} \|u\|_{\Homz}.
\label{est:H1-norm-equival}
\end{equation}
By duality, we also have
\begin{equation}
\label{est:-1nrm-equiv}
  \frac1{\sqrt{M_2}} \| \phi \|_{\Hm}\leq  \| \phi \|_{\HMm} \leq \frac1{\sqrt{M_1}} \|\phi \|_{\Hm}, \quad \forall \phi \in \Hm.
\end{equation}

With this notation, we write the energy associated to \eqref{eqn:lapm-weak}
\begin{equation}
	\tilde F(v) = \half \norm{v}{\HMo}^2 - \langle \phi, v \rangle \qquad \forall v\in\Homz.
	\label{obj:poisson-egy}
\end{equation}
Since this energy is quadratic, we can compute the optimal step size. Namely,
\begin{align}\label{eql:optimal-step}
  \ins_{n,opt}=-\frac{\left(M\nabla u_n,\nabla d_n\right) -\left\langle \phi,d_n \right \rangle }{\left(M\nabla d_n,\nabla d_n\right)}.
\end{align}
Observe also that $\tilde{F}$ is $M_1$--strongly convex and $M_2$--globally Lipschitz smooth. With these results at hand, we may apply Theorem~\ref{thm:PGD-conv} to assert convergence of PGD. The following result recasts this convergence in a manner that will be useful to establish the uniform vanishing property for perturbations.

\begin{cor}[perturbations uniformly vanish]\label{cor:unif-pertb}
  Let $u_0 \in \Homz$ and assume that the sequence $\{u_n\}_{n\geq 1} \subset \Homz$ is computed by \eqref{mthd:PGD-Pois} with a suitable $\{\ins_n\}_{n \geq 0}$, which depends on $M_1$ and $M_2$, that guarantees convergence. Then, for every $\epsilon >0$ there is $N \geq \polN$ that depends only on $\epsilon$, $M_1$, $M_2$, $\|\phi\|_{\Hm}$ and $\| u_0 \| \in \Homz$ such that
	\begin{equation}\label{est:PGD-Pois-unif-eps}
		\norm{u_{\phi,n}-u_{\phi}}{\Homz}\le \epsilon
	\end{equation}
	whenever $n\ge N$, where $u_{\phi}$ and $u_{\phi,n}$ are the solution to the Poisson equation \eqref{eqn:lapm-weak} and the $n$-th iterate of PGD computed by \eqref{mthd:PGD-Pois}.
\end{cor}
\begin{proof}
  This is nothing but the linear convergence rate of Theorem~\ref{thm:PGD-conv} together with the observations that the contraction factor, $\rho$, only depends on $M_1$, $M_2$; and $\| u \|_{\Homz}$ can be bounded by $M_2$ and $\|\phi\|_{\Hm}$.
\end{proof}

\subsection{PPGD for the stationary Cahn-Hilliard equation}

We now show how to apply the PPGD method to solve \eqref{eqn:CH3}. The main issue is that we do not wish to exactly invert the operator $-\Delta_M$. Thus, we replace it by one involving constant coefficients, which is easier to invert. This perturbation will allow us to use fast solution techniques, like FFT.

\subsubsection{Preconditioner}

As preconditioner we use
\begin{equation}\label{obj:LL}
	\LL v = \lambda(-\Delta)^{-1}v + \gamma v - \Delta v, \quad \forall \, v\in \Homz,   
\end{equation}
with $\lambda >0$ and $\gamma\ge0$. The operator $\LL$ has constant coefficients, so it can be inverted via FFT. The following are elementary results. The mapping $\LL \in \calB(\Homz,\Hm)$ is clearly a preconditioner in the sense of Definition \ref{def:precon}.

\subsubsection{Strong convexity and local Lipschitz continuity}
The energy, defined in \eqref{obj:CH-egy}, is strongly convex with locally Lipschitz continuous derivative with respect to the $\LL$--norm. To prove, this we recall (see \cite{liu1993plap}) that for $a,b\in \RR$ with $a\neq b$, $p>1$, and $\delta\ge0$ we have
\begin{align}
	\left| |a|^{p-2}a - |b|^{p-2}b\right| &\le C_3 \left| a-b\right|^{1-\delta}\left(|a|^{p-2+\delta}+|b|^{p-2+\delta} \right),
	\label{est:p-diff}
	\\
	\left( |a|^{p-2}a - |b|^{p-2}b\right)\left( a-b\right) &\ge C_4 \left| a-b\right|^{2+\delta}\left(|a|^{p-2-\delta}+|b|^{p-2-\delta} \right).
	\label{est:p-diff-ip}
\end{align}

\begin{lem}[strong convexity and local Lipschitz derivative] \label{lem:L-inf-Lip-CHegy}
	Let $\LL$ be given by \eqref{obj:LL}. The energy defined by \eqref{obj:CH-egy} is strongly convex with constant and its derivative is locally Lipschitz continuous with respect to the $\LL$--norm with constants
	\[
	  \mu \coloneqq \begin{dcases}
	                  \min\left\{\frac1{M_2 \lambda}, \frac{1-r}{\gamma C_P^2}, r \right\}, & \gamma>0,
	                  \\
	                  \min\left\{ \frac1{M_2 \lambda}, 1 \right\}, & \gamma=0,
	                \end{dcases}
    \quad
    L_B \coloneqq \max\left\{ \frac1{\lambda M_1}, 1+2C_3C_B^2C_{emb}^4 \right\},
	\]
	where $0<r<1$, $C_{emb}$ is the best constant in the embedding $\Ho \hookrightarrow L^4$ and, for $B \subset \Homz$, $C_B = \sup_{w \in B} \| w\|_{\Ho}$.
\end{lem}
\begin{proof}
	Clearly,
	\begin{align*}
		{\delta G(u) - \delta G(v)}{u-v} &=
		\iprd{\nlapmi(u-v)}{u-v} +\iprd{u^3-v^3}{u-v}
		\\
		&+\iprd{-\Delta (u-v)}{u-v}.
	\end{align*}
	Choose $0<r<1$ if $\gamma>0$ or $r=1$ if $\gamma=0$. Then, we continue the estimate
	\begin{align*}
		\pairing{\delta G(u) - \delta G(v)}{u-v} &\ge \frac1{M_2} \norm{u-v}{\Hm}^2 +\norm{u-v}{\Homz}^2
		\\
		&\ge \frac1{M_2} \norm{u-v}{\Hm}^2 +(1-r)C_P^{-2}\norm{u-v}{}^2 +r\norm{u-v}{\Homz}^2
		\\
		&\ge \mu \norm{u-v}{\LL}^2,
	\end{align*}
	where we first used \eqref{est:p-diff-ip} to get that the nonlinear term is nonnegative and \eqref{est:-1nrm-equiv}, and then the Poincar\'e inequality. This fixes the value of $\mu$.
	
	Next, we let $B \subset \Homz$ be nonempty, bounded, and convex. Pick $0 \neq w \in \Homz$, use \eqref{est:p-diff}, Holder's inequality, and Sobolev embedding to obtain, for $u,v \in B$,
	\begin{align*}
		\pairing{\delta G(u) - \delta G(v)}{w} &\le \norm{u-v}{\HMm}\norm{w}{\HMm} +\norm{u-v}{\Homz}\norm{w}{\Homz} \\
		&+C_3\int_\Omega |u-v||w|(|u|^2+|v|^2)
		\\
		&\le \norm{u-v}{\HMm}\norm{w}{\HMm} +\norm{u-v}{\Homz}\norm{w}{\Homz}
		\\
		& + 2C_3 C_{emb}^2 C_B^2\norm{u-v}{L^4}\norm{w}{L^4}
		\\
		&\le \norm{u-v}{\HMm}\norm{w}{\HMm}
		\\
		&+(1+2C_3C_B^2C_{emb}^4)\norm{u-v}{\Homz}\norm{w}{\Homz}
		\\
		&\le L_B\norm{u-v}{\LL}\norm{w}{\LL},
		\nonumber
	\end{align*}
  which fixes the value of $L_B$.
\end{proof}

\subsubsection{The PPGD}

\begin{table}[h!]
	\centering
	\begin{tabular}{@{}cc@{}}
		\toprule
		\textbf{Abstract setting \eqref{eq:GisComposite}} & \textbf{Stationary CH equation \eqref{obj:CH-egy}} \\
		\midrule
		$G$                 & $G$ given by \eqref{obj:CH-egy} \\
		$E$                 & $\frac{1}{4}\nrm{v}_{L^4}^4 + \frac{1}{2} \nrm{\nabla v }_{L^2}^2 - (f,v)_{L^2}$ \\
		$F$                 & $\frac{1}{2}\nrm{v - u_\star}_{\HMm}^2$ \\
		Preconditioner      & $\LL$ defined by \eqref{obj:LL} \\
		Approximation scheme & PGD defined by \eqref{mthd:PGD-Pois-res}--\eqref{mthd:PGD-Pois} \\
		\bottomrule
	\end{tabular}
	\caption{Correspondence between the abstract PPGD and the stationary CH equation settings.}
	\label{tab:PPGD-setting}
\end{table}

Let us now specify the components of the PPGD method \eqref{mthd:PPGD} when applied to the objective $G$, defined in \eqref{obj:CH-egy}. The perturbation $\eta_k$ will be
\begin{equation}\label{obj:pert-CH}
\etak=A(\vk - u_{\star,k}; \theta_k)-(-\Delta_M)^{-1}\left(\vk - u_{\star}\right),
	\end{equation}
where $A(\vk - u_{\star,k}; \theta_k)$ is the approximate solution of \eqref{eqn:lapm-weak} with $\phi=\vk-u_\star$ using \eqref{mthd:PGD-Pois} with parameter settings $\theta_k \in \Theta = \{(\vk, \ins,N) \} \subset \Homz \times (0,\infty) \times \polN$; i.e., $\vk$ is used as the initial guess for the approximation. Then, the perturbed residual reads
\begin{align}\label{obj:pert-res}
	\tilde r_{k} = -\dG(v_k)-\eta_{k} = f-A(\vk - u_{\star,k}; \theta_k) - v_k^3 + \Delta v_k.
\end{align}
This corresponds to the decomposition shown in Table~\ref{tab:PPGD-setting}. The method is described in Algorithm~\ref{alg:PPGD}.

\RestyleAlgo{ruled}
\SetKwComment{Comment}{/* }{ */}
\begin{algorithm}[hbt!]
	\caption{Perturbed preconditioned gradient descent method (PPGD) applied to stationary Cahn-Hilliard equation with variable mobility.}\label{alg:PPGD}
	\KwData{$f, u_\star, v_0 \neq 0$ \Comment*[r]{problem data and initial guess}}
	\KwResult{$\tilde u$  \Comment*[r]{approximate solution}}
	\textbf{Initialization}
	$\s>0$ \Comment*[r]{step size for outer loop}
	\While{$\| v_k \|_\infty > \mathrm{Tol}_{\mathrm{o}}$}{
		$\phi \gets v_k - u_\star$ \Comment*[r]{prepare inner loop}
		$v_{k,0} \gets v_k$\;
		\While{
			$\| v_{k,n} \|_\infty > \mathrm{Tol}_{\mathrm{i}}$
		}{
			$r \gets \phi + \lapm v_{k,n}$ \Comment*[r]{inner residual}
			$d \gets (-\Delta)^{-1} r$ \Comment*[r]{compute search direction}
			$\ins\gets\argmin_{\ins>0}{\tilde F(v_{k,n} + \ins d)}$ \Comment*[r]{Compute $\ins$ using \eqref{eql:optimal-step}}
			$v_{k,n+1} \gets v_{k,n} + \ins d$\Comment*[r]{descent step}
		}
		$\zeta \gets v_{k,n+1}$\Comment*[r]{store perturbation}
		$\tilde r = f - \zeta - \vk^3 + \Delta \vk$ \Comment*[r]{perturbed residual}
		$\tilde d \gets \LLinv \tilde r$ \Comment*[r]{compute search direction}
		
		$v_{k+1} \gets v_{k} + \s \tilde d$ \Comment*[r]{descent step with perturbation}
	}
	$\tilde u \gets v_{k+1}$ \Comment*[r]{output approximate solution}
\end{algorithm}
%

The convergence properties of Algorithm~\ref{alg:PPGD} are summarized below.
 
\begin{thm}[convergence of PPGD]\label{thm:PPGD-conv}
  Let $G$ be defined in \eqref{obj:CH-egy} and $u$ be its minimizer. Suppose that the sequence $\{v_k\}_{k \geq 1}$ is obtained via \eqref{mthd:PPGD} with preconditioner \eqref{obj:LL}, and the perturbed residuals \eqref{obj:pert-res} with perturbations \eqref{obj:pert-CH}.
	Then, the PPGD method converges to $u$ in the sense that, for any $\epsilon>0$, there exists $\s_0 > 0$ such that, if $\s\in(0, \s_0]$, there exist $\hat K \in\NN$ and $\{\theta_k=(\tilde{v}_k,\ins_k, N_k)\}_{k\ge0}\subset \Homz \times (0,\infty)\times \NN$ such that
	\begin{equation}
		\norm{\vk-u}{\LL}< \epsilon,
	\end{equation} 
	whenever $k \ge \hat K$.
\end{thm}
\begin{proof}
	Let the initial guess for the approximate residual be $\tilde{v}_k=\vk$.  Corollary~\ref{cor:unif-pertb} shows the uniform vanishing property for perturbations and Lemma~\ref{lem:L-inf-Lip-CHegy} shows the needed structure for $G$. Therefore, the conclusion follows from Theorem~\ref{thm:PPGD-conv-abs}.
\end{proof} 

\begin{rmk}[parameter dependencies]
  In the setting of Theorem~\ref{thm:PPGD-conv}, $\hat K$ depends on $\epsilon$, $v_0$, $C_P$, $u$ (through $B$), $C_{emb}$, $\lambda$, $\gamma$, $M_1$, and $M_2$.
	The inner number of iterations $\ins_k$ depends only on $M_1$ and $M_2$.
	The parameter $N_k$ depends on $\epsilon$, $k$, $\s$, $v_0$, $C_P$, $u$, $C_{emb}$, $\lambda$, $\gamma$, $M_1$, and $M_2$.
\end{rmk}

\section{Numerical experiments}\label{sec:numerics}

In this section, we conduct a series of numerical experiments to illustrate our analysis and the performance of the PPGD method. All codes were written in \emph{Python} in order to utilize \emph{PyTorch}'s \texttt{torch.fft.rfft} and \texttt{torch.fft.irfft} for the FFT on a graphics processing units (GPU). All these software tools are open source. We used a desktop with Intel Core i9 CPU, 64GB DDR4 3000MHz RAM, and NVIDIA GeForce RTX 3090 24GB GPU.

For spatial discretization, we employ the \emph{Fourier collocation method} \cite{canuto2007spectral, shen2011spectral}. For details we refer to \cite[Section 6.2]{psw2021PAGD}, where the authors use the same setting.

\subsection{The stationary Cahn-Hilliard equations with variable mobility}
\label{sec:num-exp2-Mob}
We compute numerical solutions to \eqref{eqn:CH3}.
The setting for the PDE and numerical method are shown in Table~\ref{tab:num-setting}.
The PDE data are as follows:

\begin{align}
	f(x,y)&=b(x,y;x_{1}, y_{1})
	\label{obj:f-L2data}
	\\
	u_\star(x,y)&=\frac{\Pi_0 b(x,y; x_{2}, y_{2})}{\norm{\Pi_0 b(\slot; x_{2}, y_{2})}{\infty}},
	\label{obj:u-star}
\end{align}
where $b(\slot; x_0, y_0):\Omega \to \RR$ is a \emph{blob} function having a global maximum at $(x_0,y_0)$:
\begin{equation}\label{obj:blob}
	b(x,y;x_0, y_0) = \exp\left( \cos\left( \frac{2\pi}{L}(x-x_0) \right)+\cos\left(\frac{2\pi}{L}(y-y_0)\right) \right).
\end{equation}
The mobility is, for $w \in \RR$,
\begin{equation}\label{obj:mob-fn}
	M(w)= \sqrt{(1-w^2)^2 + \delta_0^2},
\end{equation}
where $\delta_0>0$ is a user-defined parameter and assumed to be small. 
This is a regularized version of a thermodynamically plausible mobility $\tilde M (w)= (1-w^2) \chi_{\{|w|\leq 1\}}$, see, e.g., \cite{elliott1996degenerate}.
 
The mobility \eqref{obj:mob-fn} makes, for any measurable $w$, the composite function $\bfx \mapsto M(w(\bfx))$ satisfy \eqref{est:mob-bd}. In fact,
\begin{equation*}
	M_1 \coloneqq \delta_0 \le M(w(\bfx)) \le \sqrt{1+\delta_0^2} \eqqcolon M_2.
\end{equation*}
Hence, we have
\begin{equation}
	\frac{M_2}{M_1} \approx \delta_0^{-1}.
	\label{est:mob-ratio}
\end{equation}

\begin{table}[h!]
	\centering
	\begin{tabular}{>{\centering\arraybackslash}p{0.55\textwidth} >{\centering\arraybackslash}p{0.35\textwidth}}
		\toprule
		\textbf{PDE settings} & \textbf{Numerical settings} \\ 
		\midrule
		$\Omega = (0, 1)^2$
		
		$u_\star$ as in \eqref{obj:u-star} with $(x_1, y_1) = (0.75,\ 0.75)$,
		
		$f$ as in \eqref{obj:f-L2data} with $(x_1, y_1) = (0.25,\ 0.25)$,
		
		$M(\bfx) = \sqrt{(1 - u_\star(\bfx)^2)^2 + \delta_0^2}$,
		
		$\delta_0 = 0.1$ &
		
		$N = 2^7$, $N_0 = 1000$, $\hat{K} = 1000$,
		
		$\mathrm{Tol}_\mathrm{o} = 10^{-6}$, $\mathrm{Tol}_\mathrm{i} = 10^{-6}$,
		
		$\lambda = 1$, $\gamma = 0$,
		
		$\s = 1$, $\ins$ optimal as in \eqref{eql:optimal-step},
		
		$u_0(x, y) = 0$, $v_{k, 0} = v_k$  \\
		\bottomrule
	\end{tabular}
	\caption{Default settings of the numerical experiments for the stationary CH equation \eqref{eqn:CH3} using Algorithm~\ref{alg:PPGD}. Different values of $\delta_0$ are also used for experimentation.}
	\label{tab:num-setting}
\end{table}

Figure~\ref{fig:exp2-init} shows the plots of the PDE data functions and Figure~\ref{fig:exp2-sol} displays the plots of the numerical solution with $\delta_0=10^{-2}$ and its difference from those with $\delta_0=10^{-1},\ 10^{-3}$. As can be seen from Figure~\ref{figsub:exp2-sol-diff1} and \ref{figsub:exp2-sol-diff2}, the differences of the numerical solutions are relatively small, hence only one of them is displayed in Figure~\ref{figsub:exp2-sol}.  

\begin{figure}
	\centering
	\subfloat[$u_\star$ given by \eqref{obj:u-star} ]{\label{figsub:exp2-init-ustar} \includegraphics[width=0.5\linewidth]{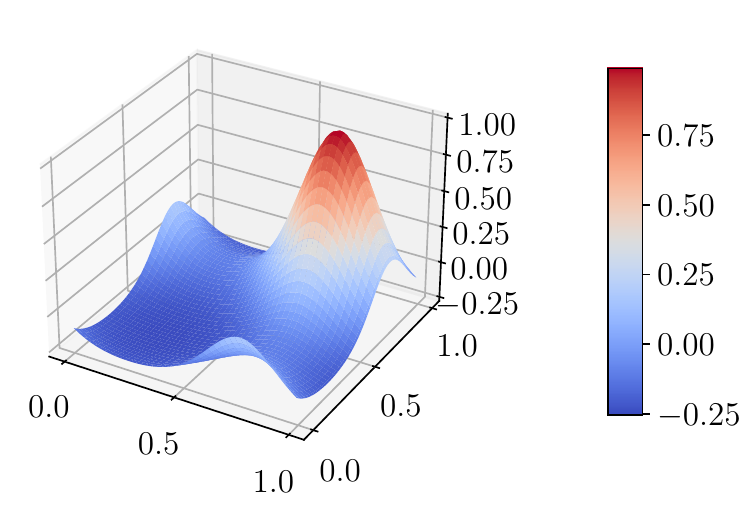}}
	\subfloat[$f$ given by \eqref{obj:f-L2data} ]{\label{figsub:exp2-init-f} \includegraphics[width=0.5\linewidth]{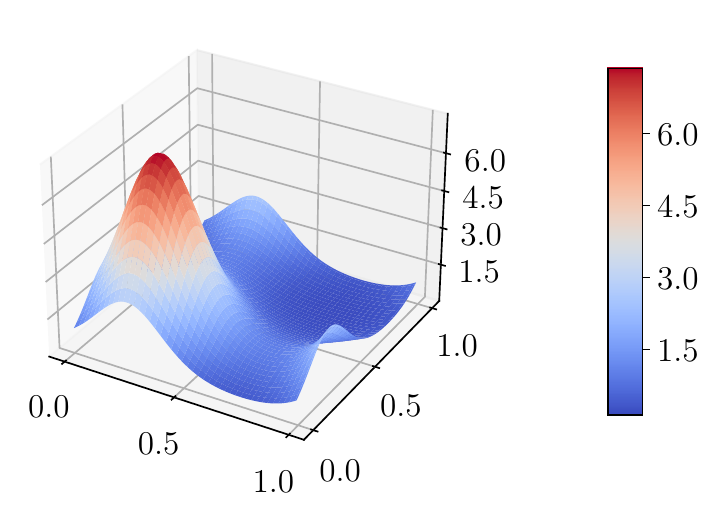}}
	\\
	\subfloat[$M$ given in Table~\ref{tab:num-setting} ]{\label{figsub:exp2-init-mob} \includegraphics[width=0.5\linewidth]{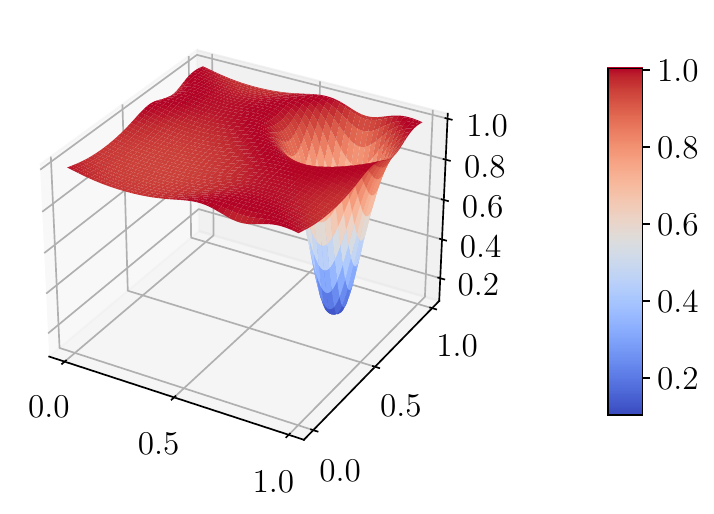}}
	\caption{Plots of the PDE data functions for numerical experiments.}
	\label{fig:exp2-init}
\end{figure}

\begin{figure}
	\centering
	\subfloat[Numerical solution \\ ($\delta_0=10^{-2}$)]{\label{figsub:exp2-sol} \includegraphics[width=0.5\linewidth]{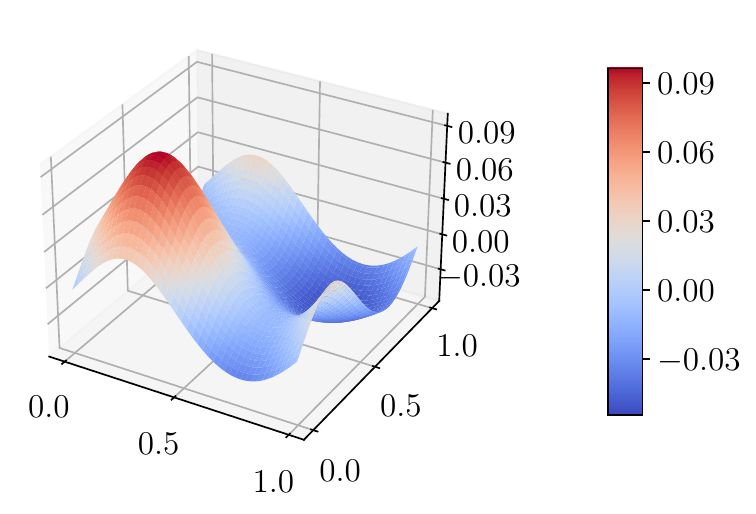}}
	\\
	\subfloat[Difference of numerical solutions \\ ($\delta_0=10^{-2}$ and $10^{-1}$)]{\label{figsub:exp2-sol-diff1} \includegraphics[width=0.5\linewidth]{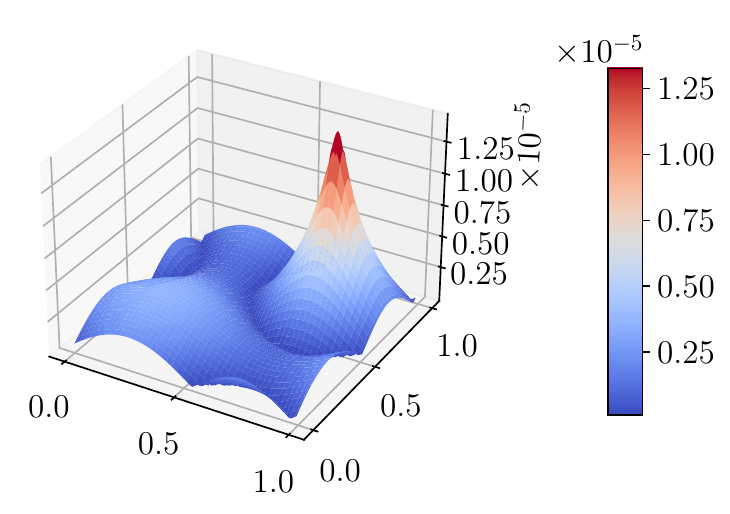}}
	\subfloat[Difference of numerical solutions \\ ($\delta_0=10^{-2}$ and $10^{-3}$)]{\label{figsub:exp2-sol-diff2} \includegraphics[width=0.5\linewidth]{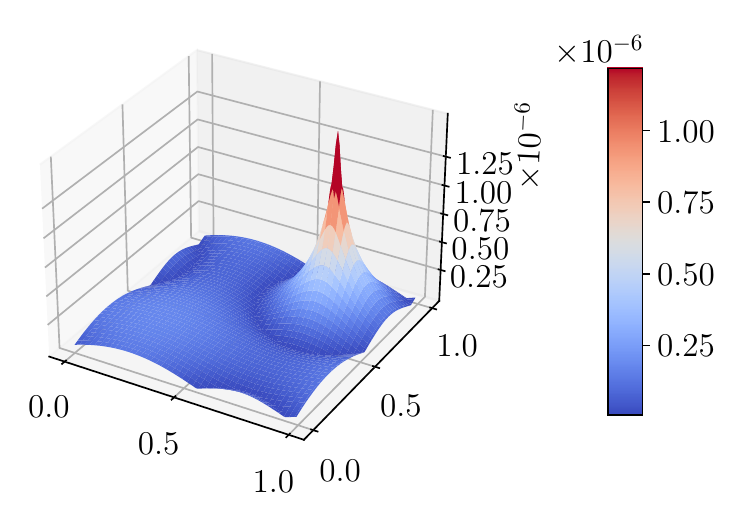}}
	\caption{Numerical solutions with different mobility ratios and their differences. The settings are given in Table~\ref{tab:num-setting}.}
	\label{fig:exp2-sol}
\end{figure}

Figure~\ref{fig:exp2-conv} shows the decay of the norm of residuals and the energy
against the number of outer iterations, and the number of inner iterations taken for each outer iteration. To compute the decay, we subtract from each data point the value of the energy at the last iterate.
Here the norm of residuals means the $\LL$--norm of the outer search direction $\norm{\tilde d}{\LL}$.
\begin{figure}
	\centering
	\subfloat[Decay of the norm of residuals ($\normL{\tilde d}$)]{\label{figsub:exp2-conv-res} \includegraphics[width=0.5\linewidth]{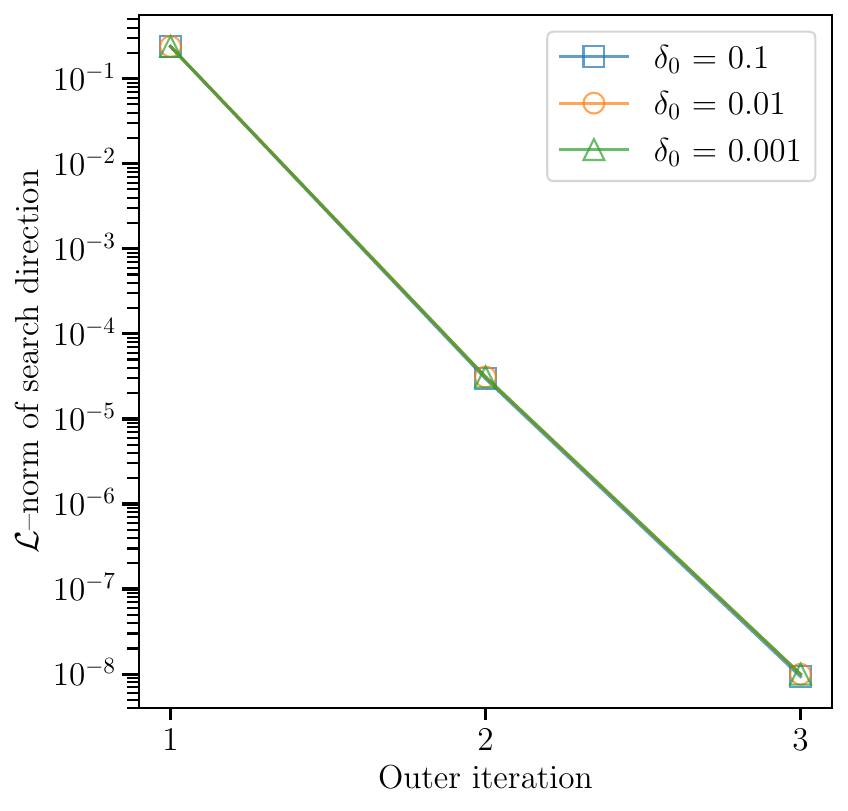}}
	\subfloat[Decay of energy ($G(v_k)$) given by \eqref{obj:CH-egy}]{\label{figsub:exp2-conv-egy} \includegraphics[width=0.5\linewidth]{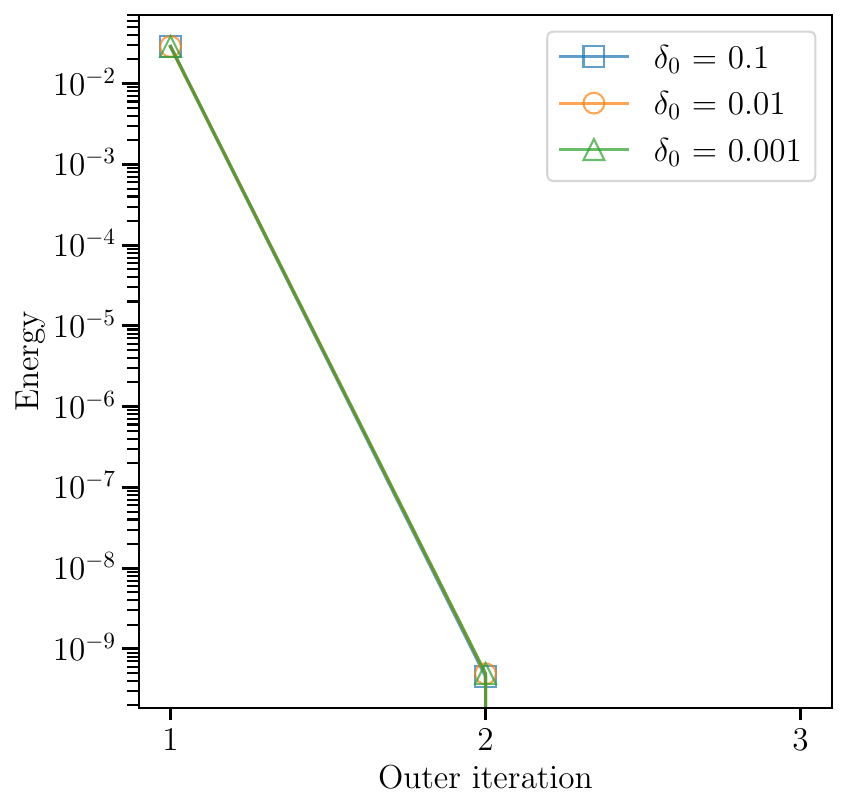}}
	\\
	\subfloat[The number of inner iterations taken for each outer iteration]{\label{figsub:exp2-conv-iter} \includegraphics[width=0.5\linewidth]{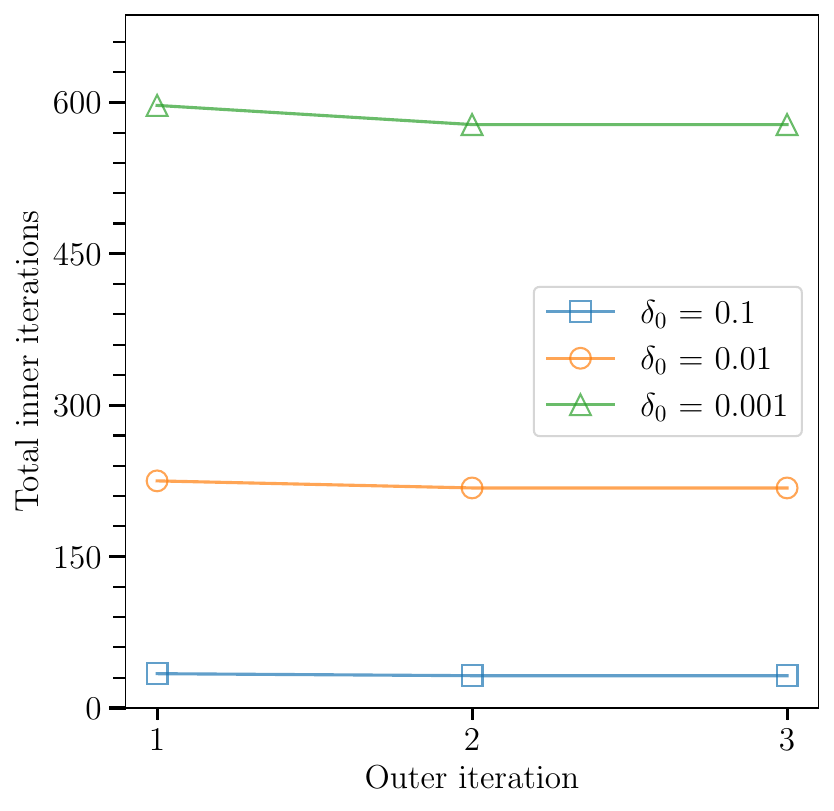}}
	\caption{Decay of the norm of residuals (a) and the energy (b), and the number of inner iterations taken for each outer iteration from experiment with different mobility ratios  (c). The settings are given in Table~\ref{tab:num-setting}.}
	\label{fig:exp2-conv}
\end{figure}

Table~\ref{tab:rslt-mob-ratio} summarizes the result of the experiments. We report the number of FFTs, the total number of inner iterations taken for the method to exit the double-loop, and the wall-clock time for each of the values of $\delta_0$ used; the smaller $\delta_0$ is, the less regular the mobility is as it can be seen in \eqref{est:mob-ratio}.  
\begin{table}[]
	\centering
	\csvreader[head to column names, 
	before reading = \begin{center},
		tabular = cccc,
		table head = \toprule $\delta_0$ & \textbf{FFT}
		& \textbf{Inner iterations} & \textbf{Wall-clock time (sec)} \\\midrule,
		table foot = \bottomrule,
		after reading = \end{center}
	]
	{MobRatio.csv}{}
	{ \tablenum[
		round-precision=3,
		round-mode=places
		]{\delnaught} & 
		\tablenum[round-precision=0, round-mode=places]{\FFT} & 
		\tablenum[round-precision=0, round-mode=places]{\iter} & 
		\tablenum[round-precision=3, round-mode=places]{\time}
	}
	\caption{Computational cost depending on the mobility ratio (Section~\ref{sec:num-exp2-Mob}). The smaller $\delta_0$ is, the less regular the mobility is as can be seen in \eqref{est:mob-ratio}.}
	\label{tab:rslt-mob-ratio}
\end{table}

Let us first discuss theoretical implications that we can see from the numerical experiments.  
Figure~\ref{figsub:exp2-conv-iter} confirms the uniform vanishing  property for perturbations as Corollary~\ref{cor:unif-pertb} asserts. Once $\mathrm{Tol}_{\mathrm{i}}$ (stopping tolerance) and $\s$ (fixed outer step size) are set, the number of inner iterations stays constant as the outer iteration proceeds for all different mobility settings. 
Figure~\ref{figsub:exp2-conv-res} and \ref{figsub:exp2-conv-egy} show that the decay of the norm of residuals and the energy almost completely overlap across different values of $\delta_0$. This suggests that the two main computational challenges are more or less independent of each other: solving the linear elliptic equation with variable coefficient, and solving the nonlinear CH equation with the former already solved, reflecting the decomposition of the residual into a linear, challenging part and nonlinear part.
The same figures also show linear convergence of the PPGD method up to perturbation tolerance as Theorem~\ref{thm:invariant-set} implies.

From the computational point of view, 
as Table~\ref{tab:rslt-mob-ratio} shows, as the ratio of the maximum and minimum of the mobility gets larger, we need more computations. Figure~\ref{fig:exp2-conv} gives a finer view of this.
We note that the double-loop of the PPGD method turns out to be fairly efficient. As Figure~\ref{figsub:exp2-conv-res} and \ref{figsub:exp2-conv-egy} show, the outer loop requires only a few iterations. This makes the efficiency of the method comparable with that of algorithms featuring only one loop.

\section{Conclusion}

We have proposed the perturbed preconditioned gradient descent method to solve composite unconstrained minimization problems, where the objective function includes a term whose exact gradient is known only approximately. The notion of a uniform vanishing property for perturbations is introduced, and we have proved convergence result for this method that is similar to other methods suited for this scenario. To be specific, when the perturbations vanish uniformly and the objective function is strongly convex with locally Lipschitz continuous derivatives, we have linear convergence, up to an error term. As a technical tool, we proved two estimates that generalize existing results to locally Lipschitz smooth functions.
These allowed us to show the existence of an invariant set for the iterations, and thus the analysis follows the one for globally Lipschitz smooth objectives, even when the method does not have a \emph{descent} property.

To illustrate the PPGD method, a numerical study of the stationary Cahn-Hilliard equations with variable mobility under periodic boundary conditions has been conducted. The method allows us to use FFT when it is not possible in usual formulations. This study has confirmed theoretical analysis, including linear convergence when the effect of perturbation seems minor. Numerical results have also given some insights into how the ratio between maximum and minimum values of the mobility affects computational efficiency.

\section*{Declaration}

\subsection*{Funding}

The work of AJS was partially supported by NSF grant DMS-2409918. The work of SMW was partially supported by NSF grant DMS-2309547. The work of JHP did not receive support from any organization for the submitted work.

\subsection*{Data availability}

The authors declare that the data supporting the findings of this study are available within the paper. Code for generating the data used can be available on \texttt{https://github.com/jhparkyb/ppgdcode.git}.

\bibliographystyle{plain}
\bibliography{PPGD.bib}

\begin{thebibliography}{10}

\bibitem{badalassi2003computation}
V.E. Badalassi, H.D. Ceniceros, and S.~Banerjee.
\newblock Computation of multiphase systems with phase field models.
\newblock {\em Journal of Computational Physics}, 190(2):371--397, 2003.

\bibitem{liu1993plap}
J.W. Barrett and W.B. Liu.
\newblock Finite element approximation of the {$p$}-{L}aplacian.
\newblock {\em Math. Comp.}, 61(204):523--537, 1993.

\bibitem{beck2009FISTA}
A.~Beck and M.~Teboulle.
\newblock A fast iterative shrinkage-thresholding algorithm for linear inverse
  problems.
\newblock {\em SIAM J. Imaging Sci.}, 2:183--202, 2009.

\bibitem{cahn1961spinodal}
J.W. Cahn.
\newblock On spinodal decomposition.
\newblock {\em Acta Metallurgica}, 9(9):795--801, 1961.

\bibitem{canuto2007spectral}
C.~Canuto, M.Y. Hussaini, A.~Quarteroni, and T.A. Zang.
\newblock {\em Spectral methods}.
\newblock Scientific Computation. Springer-Verlag, Berlin, 2006.
\newblock Fundamentals in single domains.

\bibitem{wise2020FASD}
L.~Chen, X.~Hu, and S.M. Wise.
\newblock Convergence analysis of the fast subspace descent method for convex
  optimization problems.
\newblock {\em Math. Comp.}, 89(325):2249--2282, 2020.

\bibitem{choksi2009phase}
R.~Choksi, M.A. Peletier, and J.F. Williams.
\newblock On the phase diagram for microphase separation of diblock copolymers:
  An approach via a nonlocal cahn–hilliard functional.
\newblock {\em SIAM Journal on Applied Mathematics}, 69(6):1712--1738, 2009.

\bibitem{ciarlet1989intro}
P.G. Ciarlet.
\newblock {\em Introduction to numerical linear algebra and optimisation}.
\newblock Cambridge Texts in Applied Mathematics. Cambridge University Press,
  Cambridge, 1989.
\newblock With the assistance of Bernadette Miara and Jean-Marie Thomas,
  Translated from the French by A. Buttigieg.

\bibitem{daspremont2008smooth}
A.~d'Aspremont.
\newblock Smooth optimization with approximate gradient.
\newblock {\em SIAM Journal on Optimization}, 19(3):1171--1183, 2008.

\bibitem{daubechies2003iterative}
I.~Daubechies, M.~Defrise, and C.~De~Mol.
\newblock An iterative thresholding algorithm for linear inverse problems with
  a sparsity constraint.
\newblock {\em Communications on Pure and Applied Mathematics}, 57, 2003.

\bibitem{devolder2013exactness}
O.~Devolder.
\newblock Exactness, inexactness and stochasticity in first-order methods for
  large-scale convex optimization.
\newblock 2013.

\bibitem{devolder2013firstorder}
O.~Devolder, F.~Glineur, and Y.~Nesterov.
\newblock First-order methods of smooth convex optimization with inexact
  oracle.
\newblock {\em Mathematical Programming}, 146:37 -- 75, 2013.

\bibitem{devolder2013strongconv}
O.~Devolder, F.~Glineur, and Y.~Nesterov.
\newblock First-order methods with inexact oracle: the strongly convex case.
\newblock 2013.

\bibitem{devolder2013intermediate}
O.~Devolder, F.~Glineur, and Y.~Nesterov.
\newblock Intermediate gradient methods for smooth convex problems with inexact
  oracle.
\newblock 2013.

\bibitem{douglas1956numerical}
J.~Douglas and H.H. Rachford.
\newblock On the numerical solution of heat conduction problems in two and
  three space variables.
\newblock {\em Transactions of the American Mathematical Society}, 82:421--439,
  1956.

\bibitem{elliott1996degenerate}
C.M. Elliott and H.~Garcke.
\newblock On the {C}ahn–{H}illiard equation with degenerate mobility.
\newblock {\em SIAM journal on mathematical analysis}, 27(2):404--423, 1996.

\bibitem{eyre1998unconditionally}
D.J. Eyre.
\newblock An unconditionally stable one-step scheme for gradient systems.
\newblock {\em Unpublished article}, pages 1--15, 1998.

\bibitem{feng2017PSD}
W.~Feng, A.J. Salgado, C.~Wang, and S.M. Wise.
\newblock Preconditioned steepest descent methods for some nonlinear elliptic
  equations involving p-{L}aplacian terms.
\newblock {\em J. Comput. Phys.}, 334:45--67, 2017.

\bibitem{lions1979splitting}
P.-L. Lions and B.~Mercier.
\newblock Splitting algorithms for the sum of two nonlinear operators.
\newblock {\em SIAM Journal on Numerical Analysis}, 16:964--979, 1979.

\bibitem{nesterov1983method}
Y~Nesterov.
\newblock A method for solving the convex programming problem with convergence
  rate $o(1/k^2)$.
\newblock {\em Proceedings of the USSR Academy of Sciences}, 269:543--547,
  1983.

\bibitem{nesterov2007gradient}
Y.~Nesterov.
\newblock Gradient methods for minimizing composite objective function.
\newblock {\em Research Papers in Economics}, 2007.

\bibitem{nesterov2018lectures}
Y.~Nesterov.
\newblock {\em Lectures on convex optimization}, volume 137 of {\em Springer
  Optimization and Its Applications}.
\newblock Springer, Cham, second edition, 2018.

\bibitem{nesterov2004intro}
Yu.E. Nesterov.
\newblock {\em Introductory lectures on convex optimization}, volume~87 of {\em
  Applied Optimization}.
\newblock Kluwer Academic Publishers, Boston, MA, 2004.
\newblock A basic course.

\bibitem{psw2021PAGD}
J.-H. Park, A.J. Salgado, and S.M. Wise.
\newblock Preconditioned accelerated gradient descent methods for locally
  {L}ipschitz smooth objectives with applications to the solution of nonlinear
  {PDE}s.
\newblock {\em J. Sci. Comput.}, 89(1):Paper No. 17, 37, 2021.

\bibitem{psw2023benchmark}
J.-H. Park, A.J. Salgado, and S.M. Wise.
\newblock Benchmark computations of the phase field crystal and functionalized
  {C}ahn-{H}illiard equations via fully implicit, {N}esterov accelerated
  schemes.
\newblock {\em Commun. Comput. Phys.}, 33(2):367--398, 2023.

\bibitem{peaceman1955numerical}
D.W. Peaceman and H.H. Rachford.
\newblock The numerical solution of parabolic and elliptic differential
  equations.
\newblock {\em Journal of The Society for Industrial and Applied Mathematics},
  3:28--41, 1955.

\bibitem{schmidt2011convergence}
M.W. Schmidt, N.~Le~Roux, and F.R. Bach.
\newblock Convergence rates of inexact proximal-gradient methods for convex
  optimization.
\newblock In {\em Neural Information Processing Systems}, 2011.

\bibitem{shen2011spectral}
J.~Shen, T.~Tang, and L.-L. Wang.
\newblock {\em Spectral methods}, volume~41 of {\em Springer Series in
  Computational Mathematics}.
\newblock Springer, Heidelberg, 2011.
\newblock Algorithms, analysis and applications.

\bibitem{wise2008three}
S.M. Wise, J.S. Lowengrub, H.B. Frieboes, and V.~Cristini.
\newblock Three-dimensional multispecies nonlinear tumor growth—i: Model and
  numerical method.
\newblock {\em Journal of Theoretical Biology}, 253(3):524--543, 2008.

\end{thebibliography}
\end{document}